\documentclass[12pt, A4]{article} 
\usepackage{amsmath}
\usepackage{amsfonts}
\usepackage{amsthm}
\usepackage{amscd}
\usepackage{amssymb}
\usepackage{ascmac} 
\usepackage{graphicx}
\usepackage{bounddvi} 
\usepackage[all]{xy}

\newtheorem{thm}{Theorem}[section]
\newtheorem*{thm*}{Theorem}
\newtheorem{lem}[thm]{Lemma}
\newtheorem{prop}[thm]{Proposition} 
\newtheorem{cor}[thm]{Corollary} 
\theoremstyle{defn}
\newtheorem*{defn*}{Definition}

\newtheorem*{ques*}{Question}
\newtheorem{assertion}{Assertion}
\theoremstyle{remark}
\newtheorem{rem}{Remark}
\numberwithin{equation}{section}

\newcommand{\bZ}{\mathbb{Z}}
\newcommand{\bC}{\mathbb{C}}

\newcommand{\fH}{\mathfrak{H}}

\newcommand{\bN}{\mathbb{N}}

\newcommand{\Res}{{\rm Res}}

\newcommand{\al}{\alpha}

\newcommand{\bx}{\boldsymbol{x}}

\title{The existence of $S^1\times C_p$-maps between representation spheres and its applications} 

\author{Ikumitsu {Nagasaki}\footnote{This work was supported by JSPS KAKENHI Grant Number JP23K03095. }}
\date{}

\begin{document}
%\titlepage
\maketitle
\begin{abstract}
We show the existence of $S^1\times C_p$-maps between certain representation spheres. As an application, we show that, in the family of abelian compact Lie groups, a group $G$ has the weak Borsuk-Ulam property (in the sense of Bartsch) if and only if $G$ is either a finite abelian $p$-group or a $k$-torus. 
 \end{abstract}

\section{Background and statement of results} % use lowercase except for proper names
 Some versions of Borsuk-Ulam type theorems assert the non-existence of equivariant maps. This fact has wide applications in various fields, such as nonlinear analysis, combinatorics, and discrete geometry. Details can be found in \cite{B2}, \cite{Bl} and \cite{Ma}.

In this paper, we would like to discuss the existence of certain equivariant maps.
Let $V$ be an orthogonal representation of a compact Lie group $G$, and let $S(V)$ be the representation sphere of $V$. Throughout this paper, we assume that representation spheres are $G$-fixed-point-free unless otherwise stated.
It is well-known that the Borsuk-Ulam theorem holds if  $G$ is either an elementary abelian $p$-group $C_p^k$ or a $k$-torus $T^k$, that is, if there exists a $G$-map $f:S(V)\to S(W)$, then $\dim V\le \dim W$; see, for example, \cite{FH} and \cite{M1}.
We say that a compact Lie group $G$ has the Borsuk-Ulam property if the Borsuk-Ulam theorem holds for $G$.
By results of \cite{B}, \cite{N1} and \cite{N2}, 
compact Lie groups other than $C_p^k$ and $T^k$ do not have the Borsuk-Ulam property.
Following these researches, Crabb \cite{Cr} has given explicit counterexamples to the Borsuk-Ulam theorem and provided an alternate proof in the case of finite groups.

Bartsch \cite{B} also studied Borsuk-Ulam type theorems in a weaker sense.
We say that $G$ has the {\em weak Borsuk-Ulam property} if the Borsuk-Ulam function $b_G:\bN \to \bN$, defined below, is not bounded, or equivalently, $\lim_{n\to \infty}b_G(n)=\infty$.
Here $b_G$ is defined as follows, see \cite{B}:
\begin{align*}
b_G(n)&=\max\{k\,|\,   \exists\, f: S(V)\to S(W) \text{ $G$-map with }\dim V\ge n \Rightarrow \dim W\ge k\}\\
&=\min\{\dim W\,|\,   \exists\, f: S(V)\to S(W) \text{ $G$-map with }\dim V\ge n\}.
\end{align*}
Note that $b_G$ is weakly increasing, and if there exists a $G$-map $f:S(V)\to S(W)$, then $b_G(\dim V)\le \dim W$.
By the Borsuk-Ulam theorem, $C_p^k$ and $T^k$ have the weak Borsuk-Ulam property; in fact, $b_G(\dim V) = \dim V$ for $G= C_p^k$ or $T^k$.
Bartsch \cite[Theorem 2]{B} shows that if $G$ has the weak Borsuk-Ulam property, then $G$ is a $p$-toral group, which means that $G$ has an extension 
$$1\to T^k \to G \to P\to 1,$$
where $P$ is a finite $p$-group.
Furthermore, Bartsch \cite[Theorem 1]{B} shows that, if $G$ is a finite $p$-group, then $G$ has the weak Borsuk-Ulam property. 
From this, one might expect that every $p$-toral group has the weak Borsuk-Ulam property. However, we will show that this conjecture is false. In this paper, we focus on the abelian case, particularly on the group $S^1\times C_p$. 
The main results are as follows.

\begin{thm}\label{t1-1}
The group $S^1\times C_p$ does not have the weak Borsuk-Ulam property for any prime $p$.
\end{thm}
This leads to the following corollary.
\begin{cor}\label{c1-2}
In the family of abelian compact Lie groups, a group $G$ has the weak Borsuk-Ulam property if and only if $G$ is either a finite abelian $p$-group or a $k$-torus.
\end{cor}

\section{The existence of $S^1\times C_p$-maps}
Let $S^1=\{t\in \bC\,|\ |t|=1\}$ and let $C_p$ denote the cyclic group of prime order $p$ with generator $a$. Set $G=S^1\times C_p$ and denote an element of 
$G$ by $ta^i$. 
We denote by $\bZ_n$ the cyclic subgroup of $S^1$ generated by $\xi_n=\exp(2\pi\sqrt{-1}/n)\in S^1$.
% and $C'_{p^k}=\gr{\xi_{p^k}a}$ is a subgroup of $G$ which is isomorphic to $C_{p^k}$.
The unitary (irreducible) $G$-representation $V_{k,l}$, whose underlying space is $\bC$, is defined by setting $t\cdot z= t^kz$ and $a\cdot z= \xi_p^lz$ for $z\in V_{k,l}$ ($k\in \bZ$, $l\in \bZ/p$).
%Note that $\Ker V_{p^m,1}=C'_{p^{m+1}}$, $m\ge 0$.
For positive integers $n$ and $m$, we set 
$$V_n= V_{1,1}\oplus V_{p,1}\oplus\cdots\oplus V_{p^{n-1},1}$$
 and 
$$W_m=V_{p^m,0}\oplus V_{0,1}.$$
Note that $\dim V_n =2n$ and $\dim W_m =4$.
The following result implies that $S^1\times C_p$ does not have the weak Borsuk-Ulam property.

\begin{thm}\label{t2-1}
For any integer $n\ge 1$, there exists a $G$-map $f_n:S(V_n) \to S(W_m)$ for some $m\ge n$. 
\end{thm}

If $n=1$, then one can define a $G$-map $f_1: S(V_1)\to S(W_1)$ by $f_1(z) = (z^p, 0)$ for $z\in S(V_1)$.
If $n=2$, according to \cite{Cr}, one can define a $G$-map $\bar f_2: V_2\to W_2$ by
$$\bar f_{2}(z,w)= (z^{p^2}-w^p, \overline z^{\,p}w).$$
Since $\bar f_2^{-1}(0,0)=\{(0,0)\}$,  we obtain a $G$-map $f_2:S(V_2)\to S(W_2)$ by normalizing $\bar f_2$.
\begin{rem}\label{r1}
Note that $\dim S(V_2)=\dim S(W_2)=3$ and $\deg f_2=0$, since $f_2$ is not surjective; for example, $(0,\sqrt{-1})\not\in\text{Im\,}f_2$.
More generally, it follows from \cite[Theorem 1.7]{M1} that $\deg f=0$ for any $G$-map $f: S(V_2)\to S(W_2)$.
\end{rem}

When $n\ge 3$, we show the desired result by induction on $n$.
We begin with the case of $n=3$. 
\begin{prop}\label{p2-2}
There exists a $G$-map $f_3: S(V_3) \to S(W_3)$.
\end{prop}
\begin{proof}
We decompose $S(V_3)$ into 
$S(V_3)= X_0\cup X\times I \cup X_1,$ 
where
\begin{align*}
&X_0=S(V_{1,1})\times D(V_{p,1}\oplus V_{p^2,1}),\\
&X= S(V_{1,1})\times S(V_{p,1}\oplus V_{p^2,1}),\\
&X_1= D(V_{1,1})\times S(V_{p,1}\oplus V_{p^2,1}).
\end{align*}
We define a $G$-map $g_0: X_0\to S(W_3)$ by $g_0(u,z,w)=(u^{p^3},0)$. Next, by lifting $f_2$ via $\pi:G\to G$, $ta^i\mapsto t^pa^i$, we have a $G$-map 
$$\tilde f_2: S(V_{p,1}\oplus V_{p^2,1})\to S(W_3)=S(V_{p^3,0}\oplus V_{0,1}).$$ Note that $\tilde f_2 =f_2$ as (non-equivariant) maps. We  define a $G$-map $g_1: X_1\to S(W_3)$ by $g_1(u,z,w) = \tilde f_2(z,w)$. 
Set $h_0=g_0|_{\partial X_0}: \partial X_0=X\times \{0\}\to S(W_3)$ and $h_1=g_1|_{\partial X_1}: \partial X_1= X\times \{1\}\to S(W_3)$. 
In order to obtain a $G$-map $f_3$, it suffices to show that
$h_0\coprod h_1: X\times \partial I \to S(W_3)$ can be extended to a $G$-map $H: X\times I \to S(W_3)$, that is,  $H$ is a $G$-homotopy between $h_0$ and $h_1$. In fact, if we have such a $G$-homotopy $H$, then we obtain a $G$-map $f_3=g_0\cup H\cup g_1$  by gluing $G$-maps.

Set $Y= S(V_{p,1}\oplus V_{p^2,1})$.
 We observe that there is a bijection
 $$\bar F :\text{Map}_G(X, S(W_3))\to \text{Map}_{C_p}(Y, S(W_3)),$$ which is defined by $$\bar F(\al)(z_1,z_2)= \al(1, (z_1, z_2))$$
 for $\al: X\to S(W_3)$.
 We also define a map
 $$\bar G:\text{Map}_{C_p}(Y, S(W_3)) \to
 \text{Map}_G(X, S(W_3))$$
by
$$\bar G(\beta)(u, (z_1, z_2)) =  u\cdot\beta(u^{-1}\cdot (z_1, z_2))$$
for a $C_p$-map $\beta: Y\to S(W_3)$.  One can see this as follows.
 \begin{assertion}\label{a1}
$\bar F(\al)$ is a $C_p$-map.  

 \end{assertion}
Note that $\xi_p\in S^1$ acts trivially on $Y$ and $S(W_3)$.
Then, 
\begin{align*}
\bar F(\al)(a\cdot(z_1, z_2)) &=  \al(1,\, a\cdot(z_1,z_2)) = \al(a\cdot \xi_p^{-1},\, a\cdot(z_1,z_2)) =a\cdot\al(\xi_p^{-1},(z_1, z_2))\\
&=a\cdot\al(\xi_p^{-1},\xi_p^{-1}\cdot(z_1, z_2))= a\xi_p^{-1}\cdot\al(1,(z_1, z_2))\\
&=a\cdot\al(1,(z_1, z_2)) = a\cdot \bar F(\al)(z_1, z_2).
\end{align*}

 \begin{assertion}\label{a2}
$\bar G(\beta)$ is a $G$-map.
\end{assertion}
In fact, 
\begin{align*}
\bar G(\beta)(t\cdot(u,(z_1, z_2))) &= \bar G(\beta)(tu,t\cdot(z_1, z_2))=(tu)\cdot \beta((tu)^{-1}\cdot(t\cdot(z_1, z_2))) \\
&= (tu)\cdot \beta(u^{-1}\cdot(z_1,z_2))
= t\cdot \bar G(\beta)(u,(z_1,z_2)), 
\end{align*}
and 
\begin{align*}
\bar G(\beta)(a\cdot(u,(z_1, z_2))) &=\bar G(\beta)(\xi_pu,\, a\cdot(z_1, z_2))) = (\xi_p u)\cdot \beta(u^{-1}\xi_p^{-1}a\cdot(z_1, z_2))\\
& = au \cdot \beta(u^{-1}\cdot(z_1, z_2)) =a\cdot \bar G(\beta)(u,z_1,z_2).
\end{align*}

 \begin{assertion}\label{a3}
 $\bar F^{-1}=\bar G$.
 \end{assertion}
 In fact, 
 \begin{align*}
 \bar G(\bar F(\al))(u,(z_1,z_2))&=u\cdot \bar F(\al)(u^{-1}\cdot (z_1,z_2)) \\
 &= u\cdot\al(1,\,u^{-1}\cdot (z_1,z_2))=\al(u, (z_1,z_2))
 \end{align*}
 and 
 $$\bar F(\bar G(\beta))(z_1, z_2) = \bar G(\beta)(1, z_1, z_2)
 = \beta(z_1,z_2).$$

Since a $G$-homotopy $\al_s$, $s\in I$, between $\al_0$ and $\al_1$ induces a $C_p$-homotopy $\bar F(\al_s)$ between $F(\al_0)$ and $F(\al_1)$. This is also true for $\bar G$. 
Thus $\bar F$ induces 
a bijection $$F: [X, S(W_3)]_G\to [Y, S(W_3)]_{C_p}$$
between equivariant homotopy sets.
Since $C_p$ acts freely on $Y$, 
\cite[Chapter II, \S4]{tD2} or \cite[Theorem 3.1]{Bo} shows that  
the degree of $C_p$-maps gives an injective map $$\deg: [Y, S(W_3)]_{C_p} \to [Y, S(W_3)]_{\{1\}}=[S^3,S^3]\cong \bZ.$$
Since $\bar F(h_0)$ is a constant map, it follows that $\deg F([h_0])=0$. On the other hand, $F([h_1])=[\tilde f_2]$ and $\deg \bar F(h_1)= \deg \tilde f_2=0$ by Remark \ref{r1}. Thus $\bar F(h_1)$ and $\bar F(h_1)$ are $C_p$-homotopic, and hence $h_0$ and $h_1$ are $G$-homotopic via $F$. Thus there exists a $G$-homotopy $H$ between $h_0$ and $h_1$.
\end{proof}

Next, assuming the existence of a $G$-map $f_{n-1}: S(V_{n-1})\to S(W_{m'})$ for some $m'\ge n-1\ge 3$, we show the existence of a $G$-map $f_{n}: S(V_{n})\to S(W_{m})$ for some $m\ge n$.
As in the case of $n=3$, we have a decomposition  
$$S(V_{n})= X_0\cup X\times I \cup X_1,$$ 
where
\begin{align*}
&X_0=S(V_{1,1})\times D(V_{p,1}\oplus\cdots\oplus V_{p^{n-1},1}),\\
&X= S(V_{1,1})\times S(V_{p,1}\oplus\cdots\oplus V_{p^{n-1},1})\\
&X_1= D(V_{1,1})\times S(V_{p,1}\oplus \cdots\oplus V_{p^{n-1},1}).
\end{align*}
 Similarly we have a $G$-map $g_0:X_0\to S(W_{m'+1})$ defined by $g_0(u, \bx) = (u^{p^{m'+1}}, 0)$ for $\bx\in Y=S(V_{p,1}\oplus \cdots\oplus V_{p^{n-1},1})$ and a $G$-map $g_1:X_1\to S(W_{m'+1})$
  defined by 
$g_1(u, \bx)= \tilde f_{n-1}(\bx)$, where $\tilde f_{n-1}$ is the lift of $f_{n-1}$ via $\pi: G\to G$, $ta^i\mapsto t^pa^i$.
Set $h_0=g_0|_{\partial X_0}:X\times \{0\}\to S(W_{m'+1})$ and $h_1=g_1|_{\partial X_1}:X\times \{1\}\to S(W_{m'+1})$. 
Then the following lemma holds.
\begin{lem}\label{l2-3}
There exists a $G$-map $R_{m}: S(W_{m'+1})\to S(W_{m})$ for some $m \ge m'+1$ such that $R_{m}\circ h_0$ and $R_{m}\circ h_1$ are $G$-homotopic.
\end{lem}
If we have such a $G$-homotopy $H$ between $R_{m}\circ h_0$ and $R_{m}\circ h_1$, then we obtain a $G$-map $f_{n}= R_{m}\circ g_0\cup H\cup R_{m}\circ g_1: S(V_{n})\to S(W_{m})$. Thus, once Lemma \ref{l2-3} has been proved, the proof of Theorem \ref{t2-1} will be complete. \qed

\begin{proof}[Proof of Lemma {\rm \ref{l2-3}}]
In order to show the existence of a $G$-homotopy $H$, we use equivariant obstruction theory descried in \cite[Chapter II, (3.10)]{tD2}.  
Fix a $G$-CW complex structure of $X$. 
Let denote
$$\overline{(X\times I)}_{(k)} := X_{(k-1)}\times I\cup X\times \partial I$$
be the (equivariant) $k$-skeleton relative to $X\times \partial I$ of $X\times I$. 
Assume that there exists a $G$-map on $\overline{(X\times I)}_{(k)}$ extending $h_0$ and $h_1$.
The obstruction class for the existence of a $G$-map on $\overline{(X\times I)}_{(k+1)}$ extending $h_0$ and $h_1$ lies in the equivariant cohomology groups $$\fH_{G}^{k+1}(X\times (I,\partial I), \pi_k))\cong 
\fH_{G}^k(X, \pi_k),$$ where $\pi_k=\pi_k(S(W_{m'+1}))=\pi_k(S^3)$.
Note that $G$ acts trivially on $\pi_k$,  since every element $g\in G$ acts orientation preservingly on $S(W_{m'+1})$, and it is homotopic to the identity as maps on $S(W_{m'+1})$. 
\begin{assertion}\label{a4}
$X/S^1\cong_{C_p} Y$.
\end{assertion}

In fact, we define $\varphi: X/S^1\to Y$ by $\varphi([u, \bx])=u^{-1}\cdot \bx$ for $u\in S(V_{1,1})$ and $\bx\in Y$. It is easy to see that $\varphi$ is well-defined.  Then $\varphi$ is a $C_p$-map, in fact, $\varphi(a\cdot[u,\bx]) = \varphi([\xi_pu, a\cdot\bx]) = \xi_p^{-1}u^{-1}\cdot (a\cdot\bx) = a\cdot \varphi([u,\bx]).$ Note that $\xi_p$ acts trivially on a $G$-space $Y$.
Furthermore $\psi: Y\to X/S^1$; $\bx\mapsto [1, \bx]$ is the inverse of $\varphi$.
Thus it follows that $X/S^1\cong_{C_p} Y$.

By \cite[Chapter II, \S2]{tD2} and Assertion \ref{a4}, there are isomorphisms
$$\fH_{G}^k(X;\pi_k)\cong \fH_{C_p}^k(X/S^1;\pi_k)\cong
 \fH_{C_p}^k(Y;\pi_k)\cong H^k(Y/C_p;\pi_k).$$
Since  
$C_p$ acts freely on $Y$ and $Y/C_p$ is homeomorphic to a $(2n-3)$-dimensional lens space $L^{2n-3}(p)$, 
we have
$$\fH_{G}^k(X;\pi_k)\cong
\begin{cases}
0&k=0, \ k>2n-3\\
{}_p\pi_k:=\{x\in \pi_k\,|\,px=0\}&k:\text{odd}<2n-3\\
\pi_k/p\pi_k&k:\text{even}<2n-3\\
\pi_k&k=2n-3
\end{cases}.$$
Note that $S(W_{m})$ is $2$-connected and $\fH_{G}^3(X;\bZ)=0$ when $n\ge 4$. Consequently, $\fH_{G}^k(X;\pi_k)=0$ for $k\le 3$ when $n\ge 4$. Therefore, there exists a $G$-map $\phi_4: \overline{(X\times I)}_{(4)}\to S(W_{m'+1})$  extending $h_0$ and $h_1$.
We assume that a $G$-map $\phi_{k} :\overline{(X\times I)}_{(k)} \to S(W_{m_k})$ extending $R_{m_k}\circ h_0$ and $R_{m_k}\circ \phi_k$ is inductively constructed for $k\ge 4$, where $R_{m_k}: S(W_{m'+1})\to S(W_{m_k})$ is a $G$-map for $m_k\ge m'+1$.
The obstruction class for  
an extension $\phi_{k+1}$ on $\overline{(X\times I)}_{(k+1)}$ lies in 
$$\fH_{G}^{k}(X, \pi_{k}), $$
but this class does not necessarily vanish.
Note that $\pi_{k}=\pi_{k}(S^3)$ is a finite abelian group if $k\ge 4$. We define $e_{k}$ to be the exponent of $\pi_{k}$ for $k\ge 4$.

\begin{assertion}
Let $e_k= p^ds$, where $d\ge 0$ and $(p,s)=1$.
There exists a $G$-map $Q: S(W_{m_k})\to S(W_{m_k+d})$ such that $e_k$ divides $\deg Q$.
\end{assertion}
In fact, we take integers $l$ and $r$ such that 
$sr-pl=1$, and hence $sr=1+pl$ equivalently.
Define $Q$ by $Q(z,w)= (z^{p^d}, w^{1+pl})$. It is easy to see that $Q$ is a $G$-map and $\deg Q=p^d(1+pl)=p^dsr=re_k$. Thus $e_k$ divides $\deg Q$.

Since $S^3$ is a Hopf space (in fact, a Lie group),  for any map $f: S^3\to S^3$, 
$$f_*(\al) =[f]\circ \al = ((\deg f)\iota)\circ \al= (\deg f)\al$$
 for $\al\in \pi_k(S^3)$ by \cite[Corollaries (8.4) and (8.6)]{Wh}.
Therefore, we see that $Q_* =0: \pi_k\to \pi_k$, and hence
$Q_*=0: \fH_G^{k}(X; \pi_k)\to \fH_G^{k}(X; \pi_k)$.
Thus the obstruction for $Q\circ R_k\circ\phi_k$ vanishes, and by setting $m_{k+1}=m_k+d$ and $R_{m_{k+1}} = Q\circ R_{m_k}$, we obtain a $G$-map $\phi_{k+1}$ extending $R_{m_{k+1}}\circ h_0$ and $R_{m_{k+1}}\circ h_1$.
Repeating this procedure, we obtain a $G$-homotopy $H$ between $R_{m}\circ h_0$ and $R_{m}\circ h_0$ for some $m$.
\end{proof}

\section{Conclusions and remarks}
First we prove Corollary \ref{c1-2}.
\begin{proof}[Proof of Corollary {\rm \ref{c1-2}}]
The necessary condition  has already been shown in \cite{B}. If $G$ is an abelian compact Lie group other than $C_p^k$ and $T^k$, there exists a closed subgroup $H$ such that $G/H\cong S^1\times C_p$.
By lifting $f_n$ via $\pi:G\to G/H$, we obtain a $G$-map
$\tilde f_{n}: S(\tilde V_n)\to S(\tilde W_m)$.
This implies that $b_G(n)\le b_G(2n)\le 4$, and that $G$ does not have the weak Borsuk-Ulam property.
\end{proof}

We give some remarks on the Borsuk-Ulam function.
\begin{prop}
Let $G=S^1\times C_p$. If $p$ is an odd prime, then $b_G(n)=4$ when $n\ge 4$.
\end{prop} 
\begin{proof}
We already know $b_G(n)\le 4$.
If $b_G(n) < 4$, then
$b_G(n) \le 2$, since the dimension of a fixed-point-free $G$-representation is even.
We show that there are no $G$-maps $f: S(V)\to S(W)$ if $\dim V\ge 4$ and $\dim W=2$.
In fact, we may suppose $W=V_{k,l}$ for some $(k,l)\ne (0,0)$ in $\bZ\times \bZ/p$.
Then $W^{S^1}=0$ or $W^{C_p}=0$, hence $V^{S^1}=0$ or $V^{C_p}=0$ respectively.
By the Borsuk-Ulam theorem for $S^1$ or $C_p$,  we have $\dim V\le \dim W$, which leads to a contradiction.
\end{proof}

\begin{rem}
One can see $b_G(3)=4$, $b_G(2)=2$ and $b_G(1)=2$ when $p$ is an odd prime.
\end{rem}

Next, let $G$ be a compact Lie group. The invariant $r_G(W)$ is defined by
$$r_G(W)=\sup\{\dim V\,|\,\exists\, f: S(V)\to S(W)\text{  $G$-map }\}\le\infty.$$
Here we assume that $G$-representations are fixed-point-free.

\begin{lem}\label{l3-2}
Let $H$ be a closed subgroup of $G$ and $W$ a $G$-representation $W$ with $W^H=0$. Then $r_G(W)\le r_H(W)$.
\end{lem}
\begin{proof}
Assume $r_G(W)<\infty$. Let $f: S(V)\to S(W)$ be a $G$-map with $r_G(W)=\dim V$. Then $\Res_H f: S(V)\to S(W)$ is an $H$-map with $W^H=0$.
Therefore $r_H(W) \ge \dim V=r_G(W)$.
If $r_G(W)=\infty$, then there exists a $G$-map $f: S(V)\to S(W)$ with an arbitrarily large dimension $\dim V$.  Consequently, $\Res_H f: S(V)\to S(W)$ is an $H$-map with an arbitrarily large dimension $\dim \Res_HV$. Therefore $r_H(W)=\infty$.
\end{proof}

By Bartsch's results of \cite{B}, if $G$ is not $p$-toral, then there exists a $G$-representation $W$ with $W^G=0$ such that $r_G(W)=\infty$. 
Furthermore, if $G$ has the weak Borsuk-Ulam property, then $r_G(W)<\infty$ for any $G$-representation $W$ with $W^G=0$.
On the other hand, we observe
\begin{prop}\label{p3-2}
If $G$ is a $p$-toral group, then $r_G(W)<\infty$ for every $G$-representation $W$ with $W^G=0$.
\end{prop}
\begin{proof}
Suppose that $G$ has an extension
$$1\to T^k\to G\to P\to 1,$$
where $P$ is a finite $p$-group.
By \cite[Chapter IV (3.8), 3)]{tD2}, there exists a sequence $\{P_n\}$ of finite $p$-subgroups of $G$ such that $\lim_{n\to \infty} P_n = G$ with respect to the Hausdorff metric.
For a sufficiently large $n$, it follows from $W^G=0$ that $W^{P_n}=0$. 
By Lemma \ref{l3-2}, we have $r_G(W)\le r_{P_n}(W)$.
Since a finite $p$-group has the weak Borsuk-Ulam property, we have $r_{P_n}(W)<\infty$ and thus $r_G(W)<\infty$.
\end{proof}
 
Even though the finiteness of $r_G$ holds for $G=S^1\times C_p$, this group $G$ does not have the weak Borsuk-Ulam property. One reason for this is the existence of infinitely many representations of the same dimension.

\noindent Ikumitsu Nagasaki\\
Department of Mathematics\\
Kyoto Prefectural University of Medicine\\
1-5 Shimogamo Hangi-cho\\  
Sakyo-ku 606-0823, Kyoto\\
Japan\\  
email : nagasaki@koto.kpu-m.ac.jp

\end{document}